\documentclass[conference]{IEEEtran}
\pdfoutput=1
\usepackage{cite}
\usepackage{amsmath,amssymb,amsfonts,amsthm}
\usepackage{graphicx}
\usepackage{textcomp}
\usepackage{physics}
\usepackage{mathrsfs}
\usepackage{float}
\usepackage{mathtools}
\usepackage{hyperref}

\newcommand{\E}{\mathbb{E}}
\newcommand{\expvalb}[2]{\E_{#1}\left[ #2 \right]}
\newcommand{\vecs}{\mathrm{vecs}}
\newcommand{\vecl}{\mathrm{vecl}}

\newtheorem{prop}{Proposition}
\newtheorem{lem}{Lemma}
\newtheorem{rem}{Remark}
\newtheorem{definition}{Definition}
\newtheorem{assumption}{Assumption}

\begin{document}
\title{On Solving Continuous-Discrete Projection Filters via Sum-of-Squares Relaxation}
\author{\IEEEauthorblockN{Muhammad F. Emzir}
\IEEEauthorblockA{Control and Instrumentation Engineering, IRC for Smart Mobility and Logistics\\
King Fahd University of Petroleum and Minerals, Dhahran, Saudi Arabia.
muhammad.emzir@kfupm.edu.sa}
\thanks{This work was supported in part by KFUPM-DROC, under Grant ER242009 and the IRC Smart Mobility and Logistic KFUPM under Grant INML2407 and INML2528.}}
\maketitle

\begin{abstract}
The continuous-discrete projection filter offers a rigorous framework to approximate the solution of the nonlinear state estimation problems. However, it suffers from numerical instability during the prediction phase when integration errors force the natural parameters outside their admissible domain. To address this issue, we introduce the sum-of-squares (SOS) relaxation to constrain the evolution of the natural parameters within the admissible domain. By parameterizing the underlying SOS matrix using the log-Cholesky map, we derive a projected ordinary differential equation (ODE) that inherently preserves the necessary positivity constraints without requiring the computationally expensive online optimization checks associated with previous semi-infinite programming approach. We provide a theoretical derivation of this positivity-preserving propagation scheme and present the explicit SOS-relaxed evolution equations for the Gaussian case.
\end{abstract}

\begin{IEEEkeywords}
Estimation, Stochastic Filter, Kalman Filtering, Projection Filter
\end{IEEEkeywords}

\section{Introduction}

The continuous--discrete projection filter provides an approximation to the solution of nonlinear state estimation problems in which the state dynamics evolves according to stochastic differential equations (SDEs) and measurements arrive at discrete time instants \cite{brigo1999}. Compared to its continuous-continuous counterpart, the continuous-discrete projection filter has received less attention in the literature. In our recent work, we developed a generalized continuous-discrete projection filter for exponential families with conjugate likelihood functions \cite{emzir2025}.

In the conjugate likelihood case, the projection filter enjoys certain advantages because the Bayesian update step is exact; see also \cite[Section 6]{brigo1999}. While initially developed for univariate or Gaussian families only, the projection filter for the exponential family was later extended to multivariate settings \cite{emzir2023, emzir2023a, emzir2025}. However, a critical issue with these early developments was the instability of the natural parameter propagation during the prediction phase. Specifically, the numerical approximation does not guarantee that the natural parameters remain within the admissible region. Some early attempts to address this issue employed semi-infinite programming (SIP) \cite{emzir2025b}. This approach requires modifying the ODE solver to include a checking mechanism between integration steps, applying a projection via SIP whenever the natural parameters leave the feasible set. Although the accuracy of this method depends on the number of uniform samples on the unit ball surface, the precise relationship between the local projection error and the sample count remains an open question.

In this paper, we introduce an alternative approach that maintains the natural parameters within the admissible space $\Theta$ using sum-of-squares (SOS) relaxation. Following the setup in \cite{emzir2025b}, we consider natural statistics dominated by even-order monomials. In this setting, the natural statistics can be partitioned into a lower component $c_{\ell}$ and an upper component $c_{u}$, where $c_u$ contains monomials of even order. The term ``dominated'' indicates that $\|c_\ell(x)\|/\|c_u(x)\| \to 0$ as $\|x\| \to \infty$ pointwise. Consequently, the natural parameter space decomposes as $\mathbb{R}^{m_d} \times \Theta_u$, where $\Theta_u$ consists of coefficients corresponding to negative homogeneous polynomials of even order.

The SOS approach enforces $c_u^\top \theta_u$ to be the negative of a sum-of-squares polynomial by finding a positive definite matrix that minimizes the distance to $\theta_u$. This convex optimization procedure is known as SOS relaxation \cite{lasserre2007}. Recently, Wu et al.\ \cite{wu2024} applied SOS methods to discrete-discrete nonlinear filtering for solving multi-dimensional Hamburger moment problems. In their framework, the conditional density is approximated by a surrogate density---a density divided by a positive polynomial---where the SOS method determines the positive polynomial such that the surrogate density matches the target moments exactly. Our approach differs fundamentally: rather than approximating the density directly, we use SOS relaxation to constrain the natural parameters to remain within $\Theta$.

The main contributions of this paper are as follows. First, we introduce SOS relaxation to approximate a perturbed density from a class of exponential families. This approach is equivalent to solving a convex optimization problem to find a positive definite matrix that best approximates the coefficients of the dominant natural statistics. Secondly, we project the ordinary differential equations for the natural parameter vector onto the set of SOS-relaxed natural parameters. This yields an ODE that inherently remains within $\Theta$. Unlike the semi-infinite programming procedure developed previously \cite{emzir2025b}, this approach does not require online optimization during the propagation of the natural parameters. Finally, we derive the explicit SOS-relaxed natural parameter ODEs for the Gaussian family case.

Due to the page limitation, for an extended discussion on the projection filter, we direct the reader to \cite{brigo1999} or our recent works \cite{emzir2025, emzir2023, emzir2023a}.

\section{Notation}\label{sec:Notation}
For an $A \in \mathbb{R}^{n \times n}$, we denote by $\lfloor A \rfloor$ the $n \times n$ matrix formed from the strictly lower triangular part of $A$. We denote the diagonal part of $A$ as $\mathbb{D}(A)$. For a square matrix $S$, we define a lower triangular matrix $(S)_{\frac{1}{2}} \coloneqq \lfloor S \rfloor + \frac{1}{2} \mathbb{D}(S)$. The tangent space of a smooth manifold $M$ at a point $p$ is denoted as $T_p M$. For a $C^\infty$ map $F : N \to M$ between two smooth manifolds $N$ and $M$, we denote the differential of $F$ at $p \in N$ as $F_{\ast,p} : T_p N \to T_{F(p)}M$. We refer to a homogeneous polynomial as a form; see \cite[Chapter 1]{marshall2008}, \cite[Chapter 4]{blekherman2013}. We denote the class of forms of degree $2r$ by $\mathcal{H}_{[x],r}$, while the class of sum of squared forms of degree $2r$ is denoted by $\Sigma_{[x],r}$, and the class of positive forms of degree $2r$ is denoted by $\mathcal{P}_{[x],r}$.

\section{Problem Setup and Present Challenges}\label{sec:Challenges}
Consider the state-space model
\begin{subequations}
	\label{eqs:nonlinear_SDE}
	\begin{align}
		dx_t &= f(x_t) \, dt + \varrho(x_t) \, dW_t,\\
		y_k &\sim p(y_k\mid x_k) \propto \exp(-\ell_y).
	\end{align}
\end{subequations}
where $\Delta t >0$ is a positive sampling time, $x_t \in \mathbb{R}^d$ is the state vector, $\{ W_t , t \geq 0\}$ is a Wiener process in $\mathbb{R}^{d_w}$, and $\ell_y$ is the negative log-likelihood of the discrete measurement process. The functions $f$ and $\varrho$ satisfy the standard conditions for existence and uniqueness of the SDE solution \cite{brigo1999}.

Denote the class of the probability densities $\mathcal{P}$ with respect to the Lebesgue measure on $\mathcal{X} \subseteq \mathbb{R}^d$ as $\mathcal{P} = \{p \in L^1 : \int_\mathcal{X} p(x) \, dx = 1, p(x) \geq 0, \forall x \in \mathcal{X}\}$. The exponential family with natural statistics $c:\mathbb{R}^d \to \mathbb{R}^m$ is defined as $\text{EM}(c) \coloneqq \left\{ p \in \mathcal{P} \colon p(x) = \exp(c(x)^\top\theta - \psi(\theta)) \right\}$, where $\theta \in \Theta \subset \mathbb{R}^m$ is the natural parameter. The natural statistics in $c$ are assumed to be linearly independent. The corresponding natural parameter space $\Theta$ is given by $\Theta \coloneqq \left\{ \theta \in \mathbb{R}^m\colon \int_\mathcal{D} \exp(c(x)^\top \theta) \, dx < \infty \right\}$,
where $\mathcal{D}\subseteq \mathcal{X}$ is the support of $\exp(c(x)^\top \theta)$. Among many exponential families, we are interested in the regular exponential family which has its natural parameter space $\Theta$ as an open subset of $\mathbb{R}^m$. One of the virtues of the exponential family $\text{EM}(c)$ is that the expected value of $c$ with respect to $p_\theta \in \text{EM}(c)$ and its covariance can be computed via the cumulant-generating function (i.e., the log Laplace transform or log partition function \cite{brown1986,emzir2023}). The cumulant-generating function is defined by
\begin{equation}
  \psi(\theta) = \log \left[ \int_\mathcal{D} \exp(c(x)^\top \theta) \, dx \right], \quad \theta \in \Theta. \label{eq:cumulant-generating}
\end{equation}

On the continuous-discrete setting, the exponential family projection filter amounts to propagating the natural parameters in two steps: the prediction step, where the projected square-root of the Fokker--Planck equation is propagated, and the update step, where the measurement is used to update the natural parameters. The prediction step requires solving the following ordinary differential equation (ODE) for the natural parameters:
\begin{equation}
  \dv{\theta_t}{t} = g(\theta_t)^{-1}\E_{\theta} \left[\mathcal{L}(c)\right], \quad t \in (k\Delta t, (k+1) \Delta t), \label{eq:Projected_Fokker_Planck}
\end{equation}
where $g(\theta)$ is the Fisher information metric, and $\mathcal{L}(\varphi)
  = \sum_{i=1}^df_i(x)\,\pdv{\varphi(x)}{x_i}
  + \tfrac{1}{2}\sum_{i,j=1}^d (\varrho \varrho^\top)_{ij}(x)\,\pdv[2]{\varphi(x)}{x_i}{x_j}$ is the backward Kolmogorov diffusion operator.

Since the numerical integration is not exact, the increment $\dv{\theta}{t}$ calculated by \eqref{eq:Projected_Fokker_Planck} at each step is only an approximation. Therefore, even in the case of exact Bayesian update as is the case in the conjugate exponential family, the local error in the projection filter predictive steps is present and consists of the quadrature errors in the numerical integration and the projection error due to projecting the $L^2(\mathcal{X})$ vector onto the tangent space (to be identified in the next section). These two local error sources are then accumulated when using the ODE solver to solve \eqref{eq:Projected_Fokker_Planck}.

Upon solving \eqref{eq:Projected_Fokker_Planck}, due to the integration errors of the ODE solver and the sparse-grid, the vector of natural parameters $\theta$ might leave the space of natural parameters $\Theta$. In practice, when this occurs, it can be identified by the value of $c(x)^\top \theta$ in \eqref{eq:cumulant-generating} that does not go to $-\infty$ as $\|x\|\to \infty$. When this function is evaluated on the quadrature nodes, a small fraction of the quadrature nodes will have extremely high values. This leads to an invalid value of $\psi(\theta)$. As a consequence, the Fisher information metric $g(\theta)$ is no longer positive definite which eventually ends the propagation of  \eqref{eq:Projected_Fokker_Planck}.
\section{Proposed Methods}\label{sec:Proposed_Method}

Having established the problem setting and the challenges associated with maintaining feasibility of the natural parameters, we now present our proposed solution based on sum-of-squares relaxation.

\subsection{Sum-of-Squares Relaxation}\label{sec:SOS_Projection}

In this section, we propose an approximation method based on sum-of-square, a well-known method in polynomial optimization \cite[Chapter 2]{lasserre2010}. Consider the following set of monomials with order equal to $r \in \mathbb{N}$:
$
  w_r (x) = \left[ x^\alpha \right]_{\alpha \in \mathbb{N}^d, \abs{\alpha} = r}.
$
Shortly speaking, in the SOS relaxation method, given a positive homogeneous polynomial of degree $2r$, $v(x)$, we seek a positive definite matrix $P$ such that $w_r^\top P w_r$ is as close as possible to $v(x)$. In 1888, Hilbert found that, $\mathcal{P}_{[x],r} = \Sigma_{[x],r}$ only holds for $r=1$ with any $d$, $d=1$ and $2$ with any $r$, and $d=3$ with $r=2$; \cite{hilbert1888}, \cite[Theorem 1.2.6]{marshall2008}. Other than these cases, $\Sigma_{[x],r} \subsetneq \mathcal{P}_{[x],r}$. Let $\mathcal{X}=\mathbb{R}^d$. As we briefly discussed in the previous section, the natural statistics $c$ is to be divided into two sections, $c_u: \mathbb{R}^{d} \to \mathbb{R}^{m_u}$ and $c_l : \mathbb{R}^{d} \to \mathbb{R}^{m_l}$, where $m_u + m_l = m$ and $c_u=w_{2r}$. Correspondingly, we also divide the natural parameters into $\theta^\top = [\theta_l^\top,\theta_u^\top]$. For any $v \in \mathcal{H}_{[x],r}$, due to the homogeneity, we have $v(x) = v(\|x\| u) = \|x\|^{2r} v(u)$ where $u = x/\|x\| \in \mathcal{S}^{d-1}$ (the unit sphere in $\mathbb{R}^d$). The ``direction" $u$ on the unit sphere completely determines the sign of $v(x)$. We set the following additional assumption on the natural statistics:

\begin{assumption}\label{asm:c_u_dominated_c_l}
The natural statistics vector $c_u$ is selected as $w_{2r} \in \mathcal{H}_{[x],r}$ for some $r \in \mathbb{N}$ such that all elements in $c^\top  = [c_u^\top, c_l^\top]$ are linearly independent and $\lim_{\|x\| \to \infty} \frac{\|c_l(x)\|}{\|c_u(x)\|} = 0$ pointwise.
\end{assumption}

Assumption \ref{asm:c_u_dominated_c_l} simplifies the verification of a $\theta$ to be in the natural parameter space $\Theta$ by ensuring $c_u^\top \theta_u$ is a negative form. In particular, it guarantees that constraining only $c_u^\top \theta_u < 0$ on the unit sphere is \emph{sufficient} to ensure $\theta \in \Theta$, since the lower-order terms $c_d$ cannot cause divergence of $\int \exp(c^\top \theta)\, dx$. Without this partition, one would need to characterize the full natural parameter space $\Theta$, which lacks a tractable description for general exponential families. This assumption is satisfied whenever the highest-order terms in the natural statistics are even-degree monomials that grow faster than all other terms---for example, the univariate Gaussian family where $c = [x, x^2]$ with $c_u = x^2$.

Suppose for some $\theta_u \in \mathbb{R}^{m_u}$, $-c_u^\top \theta_u \notin \Sigma_{[x],r}$. We are interested in finding a negative of an SOS form $v$
\begin{gather}
\begin{aligned}
v &= \sum_{|\alpha| = 2r} \hat{v}_\alpha x^\alpha
&= -w_r (x)^\top P w_r (x)
  \end{aligned}\label{eq:p_neP_wr_Q_wr}
\end{gather}
such that it best approximates $c_u^\top \theta_u$. As it is clear from the dimensionality, the polynomial coefficient of $v$, $\hat{v} \in \mathbb{R}^{s_h (2r, d)}$, and $P \in \mathbb{R}^{s_h (r,d) \times s_h (r,d)}$ where $s_h (m,n) \coloneqq \binom{m+n-1}{n-1}$; \cite{marshall2008}. Therefore,
$
\tfrac{(s_h (r, d)^2)}{s_h (2r, d)} = \prod_{j=1}^{d-1} \tfrac{r^2 + 2 r j + j ^2}{j ^2 + 2 r j} > 1,
$
which means the matrix $P$ in \eqref{eq:p_neP_wr_Q_wr} contains more elements compared to $\hat{v}$. Hence, in this representation the $P$ is not unique, as there are infinite possibilities of solution that satisfy \eqref{eq:p_from_G}. The matrix $P$ above can always be chosen as a symmetric matrix. The approximation of $-c_u^\top \theta_u$ by an element from $\Sigma_{[x],r}$ can be formulated as finding a $P \succ 0$ such that $v = -w_r ^\top P w_r$ is as close as possible to $c_u^\top \theta_u$. For this work, we measure the distance of two polynomials by the Euclidean norm of the difference of their coefficients; e.g., for $v_1,v_2 \in \mathcal{H}_{[x],r},v_1 = \hat{v}_1^\top w_{2r},v_2=\hat{v}_2^\top w_{2r} $, we define $\norm{v_1 - v_2}_{\mathcal{H}_{[x],r}} = \norm{\hat{v}_1-\hat{v}_2}_2$.

In order to compare the coefficient of a positive form $v$ and its SOS approximation, we need to be able to convert the positive definite matrix $P$ in  \eqref{eq:p_neP_wr_Q_wr} to the corresponding polynomial coefficient. Let $z_1 \coloneqq s_h (r,d)$ and $z_2 \coloneqq s_h (2r,d)$. If we select $P$ to be symmetric, then $P$ is fully determined by its lower triangular elements. Let $\mathrm{vec}: \mathbb{R}^{n\times m} \to \mathbb{R}^{n m}$ be the operator that stacks an $\mathbb{R}^{n\times m}$ matrix into a column vector, $\mathrm{lt}(n) = \frac{1}{2}(n (n+1))$, and $\vecs: \mathbb{R}^{n \times n} \to \mathbb{R}^{\mathrm{lt}(n)}$ be the mapping that maps the diagonals of its lower triangular elements to a vector; i.e., $\vecs(P) \coloneqq [d_{P,0}^\top, d_{P,1}^\top, \hdots, d_{P,n-1}]^\top,$ where $d_{P,0}$, is the main diagonal of $P$, and $d_{P,i}$ is the $i$th sub-diagonals of $P$, respectively. We can write the coefficient vector of the negative SOS form $v = -w_r^\top P w_r$ as
\begin{equation}
  \hat{v}(P) \coloneqq -T_p \vecs(P), \label{eq:p_from_G}
\end{equation}
where $T_p \in \mathbb{R}^{z_2 \times \mathrm{lt}(z_1)}$. The following proposition explicitly gives $T_p$.
\begin{prop}\label{prp:T_p}
  The linear transformation matrix $T_p$ in \eqref{eq:p_from_G} is given by:
  \begin{equation}
    T_p = (T T^\top)^{-1} T K,
  \end{equation}
  where $K \in \mathbb{R}^{ \mathrm{lt}(z_1) \times  \mathrm{lt}(z_1)}$ is a diagonal matrix with the diagonal elements given by
  \begin{equation}
    K_{k,k} = \begin{cases}
      1 &, ~\quad 1 \leq k \leq z_1,\\
      2 &, ~\quad k > z_1.
    \end{cases}\label{eq:K_diagonals}
  \end{equation}
  and,
  $T \in \mathbb{R}^{z_2 \times \frac{1}{2} (z_1 (z_1+1))}$ satisfying
  \begin{gather}
    \begin{aligned}
    T_{l,k} =& \begin{cases}
      1 &, ~\quad \alpha_l = \beta_i + \beta_j, k=k_0 \\
      0 &, ~\quad \text{otherwise},
    \end{cases}\\
    k_0 =& \begin{cases}
        i &, i=j\\
        z_1 + (i-j-1) (z_1 - \frac{i-j}{2}) + j&, i>j
      \end{cases}.
  \end{aligned}\label{eq:T_explicit}
  \end{gather}
  where $x^{\beta_i+\beta_j}$ is the $i,j$ component of  $(w_r w_r^\top)$  and $x^{\alpha_l}$ is the $l$-th component of $w_{2r}$.
\end{prop}
\begin{proof}

Let $w_{r^2 \mathrm{sym}} \coloneqq \vecs(w_r w_r^\top)$. Since $P$ is symmetric, $\mathrm{vec}(P)^\top \mathrm{vec}(w_r w_r^\top) = \vecs(P)^\top K w_{r^2 \mathrm{sym}}$ for some diagonal $K$. Separating diagonal and off-diagonal terms gives
$
  \sum_i^{z_1} P_{ii}w_{r,i}^2 + \sum_{i>j}^{z_1} P_{ij}w_{r,i}w_{r,j} = \sum_{i\geq j} k_{i,j} P_{ij}w_{r,i}w_{r,j}.
$
Equating both sides, we get $K_{i,i}=1$ and $K_{i,j}=2$ if $i>j$. This implies \eqref{eq:K_diagonals}.
Next, consider
$
  w_{2r} = T w_{r^2 \mathrm{sym}}.
$
Similarly, equating the element on the left and the right, we obtain \eqref{eq:T_explicit}.
Now, for symmetric $P$, $-\hat{v}^\top w_{2r} = w_r^\top P w_r = \mathrm{vec}(P)^\top \mathrm{vec}(w_r w_r^\top)$.
Using \eqref{eq:K_diagonals} and \eqref{eq:T_explicit}, $\hat{v}^\top w_{2r} = \hat{v}^\top T w_{r^2 \mathrm{sym}}$ and $w_r^\top P w_r = \vecs(P)^\top K w_{r^2 \mathrm{sym}}$, so $\vecs(P)^\top K = -\hat{v}^\top T$. Hence, taking a transpose on both sides and multiplying both sides with $T$ from the left, we get $T K^\top \vecs(P) = - T T^\top \hat{v}$. The matrix $T$ is full row-rank, and hence $T T^\top$ is invertible; see \cite[\textsection 0.4.6]{horn2012}. Therefore, we obtain $- (T T^\top)^{-1} T K^\top \vecs(P) = \hat{v}$. Hence, $T_p = (T T^\top)^{-1} T K^\top$ as required.
\end{proof}

Using $T_p$ given in the proposition above, the distance between $\theta_u^\top c_u$ and $-w_r^\top P w_r$
is given by $\norm{\theta_u^\top c_u + w_r^\top P w_r}_{\mathcal{H}_{[x],r}}
= \norm{\theta_u + \hat{v}(P)}_2$.
Therefore, for any given $\theta_u \in \mathbb{R}^{m_u}$, the negative-SOS coefficient is given by solving the following semi-definite programming:
\begin{align*}
  \min & \norm{\theta_u + \hat{v}(P)}_2 \\
  \text{subject to }& P \succ 0.
\end{align*}
Once we found the optimizer $P$, then the approximated form is then given by $v = \hat{v}(P)^\top w_{2r}$.

\subsection{Propagating SOS Matrix via Log-Cholesky Coordinates}\label{sec:Propagating_SOS_via_LogCholesky}

In this section, we will derive an ordinary differential equation that ensures the evolution of the natural parameters always stays in the SOS-relaxed space:
\begin{equation}
  \Theta_{SOS} \coloneqq \left\{ \theta \in \mathbb{R}^m:  \theta_l \in \mathbb{R}^{m_l}, \theta_u = -T_p \hat{v}(P); P \succ 0 \right\}
\end{equation}
While we can derive explicitly the ODE of $P_t \in \mathcal{S}_{z_1}^+$ for given $\theta_{u,t}$, the solver for this ODE needs to abide the manifold structure of $\mathcal{S}_{z_1}^+$. Standard ODE solvers such as Runge--Kutta are known to fail to preserve the positivity of the solution of ODE in $\mathcal{S}^+_{z_1}$. Apart from traditional ODE solvers, there is a class of geometric ODE solvers that respect this $\mathcal{S}^+_{z_1}$ structure like \cite{drumetz2024,munthe-kaas2023}. Instead of using these geometric ODE solvers, we will solve the problem by using the log-Cholesky chart introduced in \cite{lin2019}.

Suppose at time $t_k = k \Delta t$, we have obtained $\theta_{u,k}$ such that $c_u^\top \theta_{u,k}$ is a negative SOS form, or if it is not the case, then we apply the projection method detailed in the previous section. Therefore, there exists a positive definite matrix $P_k$ such that $c_u^\top \theta_{u,k} = - w_r^\top P_k w_r$. Between the observation time, i.e., for $t \in [t_k, t_{k+1})$, we propagate the natural $\theta$ via \eqref{eq:Projected_Fokker_Planck}.
Let $P_t$ be the SOS matrix that satisfies $c_u^\top \theta_{u,t} = - w_r^\top P_t w_r$, where at $t_k=k\Delta t$, $P_{t_k} = P_k$. If $c_u^\top \theta_{u,k}$ is the negative of an SOS form, then from Proposition \ref{prp:T_p}, we have $\theta_{u,k} = -T_p \vecs(P_k)$. Hence, $P_k = -\vecs^{-1}(T_p^+ \theta_{u,k})$.
Therefore,
\begin{equation}
  \frac{d P_{t}}{d t} = -\vecs^{-1}\left(T_p^+ \begin{bmatrix}0_{m_u \times m_l} & I_{m_u}\end{bmatrix} \frac{d \theta_t}{d t}\right)\label{eq:dP_dt}
\end{equation}
 Let $\mathscr{L} :  \mathcal{S}^+_n \to \mathcal{L}^+_n$ be the Cholesky map that maps a positive definite matrix $P$ to its Log-Cholesky representation. It is well-known that the Cholesky map is a diffeomorphism between the cone of positive definite matrices $\mathcal{S}^+_n$ and the cone of lower triangular matrices with positive diagonal elements $\mathcal{L}^+_n$; see \cite{lin2019}. In order to ensure that $P_t$ remains positive definite during the propagation, we use the global mapping introduced in \cite{lin2019} that maps its Cholesky factor $L_t$ to $\mathbb{R}^{\frac{1}{2}n(n+1)}$ via its Log-Cholesky map as follows \cite[Proposition 3]{lin2019}
\begin{equation}
  \vecl(L) \coloneqq \begin{bmatrix}
    \log(d_{L,0})^\top&
    d_{L,1}^\top&
    \hdots&
    d_{L,n-1}^\top
  \end{bmatrix}^\top
\end{equation}
The map $\vecl$ stacks the main diagonal $d_{L,0}$, and subdiagonals $d_{L,i}, i=1,\ldots, n-1$, of the Cholesky factor $L$ of $P$ and applying log map only to the main diagonal elements of $L$.  Let the log-cholesky coordinate of $P_t$ be denoted by $\nu = \vecl \circ \mathscr{L}$.  The time derivative of the log-Cholesky coordinates of $P_t$ is given by $\frac{d \nu_t}{d t} = \vecl_{\ast,L_t} \circ \mathscr{L}_{\ast,P_t} \circ \frac{d P_t}{d t}$, where $\nu_t \coloneqq \nu(P_t), L_t \coloneqq \mathscr{L}(P_t)$; i.e., $\nu_t = \vecl(L_t)$. The Jacobian of $\vecl$ can be written explicitly as
  \begin{equation*}
    \vecl_{\ast,L} = \mathrm{diag}\left(\begin{bmatrix}
      \frac{1}{d_{L,0}}\\
      1_{\mathrm{lt}(n)-n}
    \end{bmatrix}\right).
  \end{equation*}
where $\frac{1}{d_{L,0}}$ is the reciprocal of the main diagonal of $L$, and $1_{\mathrm{lt}(n)-n}$ is a column vector of length $\mathrm{lt}(n)-n$ with all elements equal to one. Therefore, instead of propagating \eqref{eq:Projected_Fokker_Planck}, we propagate the following ODE between the observation time:

\begin{subequations}
\begin{align}
\frac{d \theta_{l,t}}{d t} =& \begin{bmatrix}I_{m_l} & 0_{m_l \times m_u}\end{bmatrix} g(\theta_t)^{-1} \frac{d \eta_t}{d t}\\
\frac{d \nu_t}{d t} =& \vecl_\ast (L_t) \circ \mathscr{L}_{\ast}(P_t) \circ \frac{d P_t}{d t}\label{eq:dnu_dt_original}
\end{align}
\end{subequations}
where $\frac{d \eta_t}{d t} \coloneqq \expvalb{\theta_t}{\mathcal{L}[c]}$, and $\frac{d P_{t}}{d t}$ is computed via
\begin{equation}
  \frac{d P_{t}}{d t} = -\vecs^{-1}\left(T_p^+ \begin{bmatrix}0_{m_u \times m_l} & I_{m_u}\end{bmatrix} g(\theta_t)^{-1} \frac{d \eta_t}{d t}\right)\label{eq:dP_dt_explicit}
\end{equation}
Since $\frac{d P_{t}}{d t}$ \eqref{eq:dP_dt_explicit} involves $\vecs^{-1}$, and \eqref{eq:dnu_dt_original} involves $\vecl$, the ODE \eqref{eq:dnu_dt_original} can be given explicitly as function of $\tfrac{d \eta_t}{d t}$, which is our next objective. For clarity in the following theoretical development we introduce the following definitions and lemmas.
\begin{definition}\label{def:Dn_Mn_id}
  \begin{enumerate}

    \item The duplication matrix of dimension $n$, $D_n \in \mathbb{R}^{n^2 \times \mathrm{lt}(n)}$ is a rectangular matrix that satisfies $D_n \vecs(A) = \mathrm{vec}(A)$ for any $A \in \mathcal{S}_n$.
    \item The embedding matrix of dimension $n$, $E_n \in \mathbb{R}^{n^2 \times \mathrm{lt}(n)}$ is a rectangular matrix that satisfies $E_n \vecs(X) = \mathrm{vec}(X)$ for any $X \in \mathcal{L}_n$.
    \item The commutation matrix $M_n  \in \mathbb{R}^{n^2\times n^2}$ is a matrix that satisfies $M_n \mathrm{vec}(A)=\mathrm{vec}(A^\top)$.
    \item The index mapping $\mathrm{id} : \left\{ 1,\ldots, \mathrm{lt}(n)  \right\} \to \left\{ (i,j): 1\leq j \leq i \leq n \right\}$ denotes the inverse of the $\vecs$ indexing.
\end{enumerate}
\end{definition}
Since $D_n$ is full rank, then $D_n^+ = (D_n^\top D_n)^{-1}D_n^\top \in \mathbb{R}^{\mathrm{lt}(n)\times n^2}$, and $D_n^+ \mathrm{vec}(A)=\vecs(A)$. The same also holds for $E_n$.
\begin{lem}\label{lem:vecs_kroneker}
  The following statements hold for any $L,X \in \mathcal{L}_n$ and $S \in \mathcal{S}_n$
  \begin{enumerate}
    \item
  \begin{equation}
    \vecs(L X) = E_n^+ \left( I_n \otimes L \right)E_n \vecs(X).\label{eq:vecs_LX}
  \end{equation}
  \item
  \begin{equation}
    \vecs(L S L^{\top}) = D_n^+ \left( L \otimes L \right)D_n \vecs(S).\label{eq:vecs_LSL}
  \end{equation}
  \end{enumerate}
\end{lem}
\begin{proof}
  In vectorized form:
  \begin{equation*}
    \mathrm{vec}(L X) = \left( I_n \otimes L \right)\mathrm{vec}(X)
  \end{equation*}
  Eq. \eqref{eq:vecs_LX} then follows by the definition of $E_n$. Eq. \eqref{eq:vecs_LSL} can be obtained similarly  using \cite[Proposition 7.19]{bernstein2009}
  \begin{equation*}
     \mathrm{vec}(L S L^\top) = \left( L \otimes L \right)\mathrm{vec}(S).
  \end{equation*}
\end{proof}
\begin{lem}\label{lem:d_nu_dt_explicit}
  Given $L \in \mathcal{L}^+_n$, define the congruence matrix of $L$, as $C(L)$:
  \begin{equation}
    C(L) \coloneqq D_{n}^+\left( L^{-1}\otimes L^{-1} \right)D_{n}.
  \end{equation}
  Further, define $\Delta(L) \coloneqq \vecl_{\ast,L} E_{n}^+( I_{n} \otimes L )E_{n}$ and $\tilde{K}_{k,k} = \frac{1}{2}$ for $k \leq n$, $\tilde{K}_{k,k} = 1$ for $k > n$.
  If $P_t \in \mathcal{S}^+_n$ has Cholesky factorization $L_t$ for any time $t$, and $\tfrac{d P_t}{dt} = -\vecs^{-1}(\xi_t)$, then the log-Cholesky coordinate $\nu_t = \vecl\circ \mathscr{L}$ follows the following ODE:
  \begin{equation}
    \dfrac{d \nu_t}{dt} = -\Delta(L_t)\tilde{K}C(L_t)\xi_t. \label{eq:dnu_dt}
  \end{equation}
\end{lem}
\begin{proof}
  Using \cite[Proposition 4]{lin2019}, with $S_t = L^{-1}\dfrac{d P_t}{dt}L^{-\top}$, the time derivative of $L_t$ is given by
  $
    \dfrac{d L_t}{dt} = L_t \left(S_t\right)_{\frac{1}{2}}.
  $
  Using Lemma \ref{lem:vecs_kroneker}, we get
  \begin{align*}
    \vecs\left(\dfrac{d L_t}{dt}\right) =& E_{n}^+\left( I_{n} \otimes L_t \right)E_{n}\vecs(\left(S_t\right)_{\frac{1}{2}})\\
    \vecs(S_t) =& D_{n}^+\left( L^{-1}\otimes L^{-1} \right)D_{n} \vecs\left(\frac{d P_t}{dt}\right).
  \end{align*}
  Moreover,  $\vecs((S_t)_{\frac{1}{2}}) = \tilde{K}\vecs(S_t)$. Hence, we obtain:
  \begin{equation*}
    \vecs\left(\dfrac{d L_t}{dt}\right) =  -E_{n}^+\left( I_{n} \otimes L_t \right)E_{n}  \tilde{K} D_{n}^+\left( L_t^{-1}\otimes L_t^{-1} \right)D_{n}\xi_t
  \end{equation*}
  Using this equation, and the definition of $\vecl_\ast$, we obtain \eqref{eq:dnu_dt}.
\end{proof}

Using the lemmas above, we can finally have the following result:
\begin{prop}\label{prop:theta_l_nu_ode}
Let $\theta_k$ be the natural parameters at time $t_k = k \Delta t$, for some $k \in \mathbb{N}$, and measurement sampling time $\Delta t$. Then the sum-of-square relaxation of the natural parameters $\theta_t$ between two consecutive measurement sampling times $t \in (k \Delta t, (k+1)\Delta t)$ can be written compactly as
\begin{align}
  \dfrac{d}{dt} \begin{bmatrix}
    \theta_{l,t}\\
    \nu_t
  \end{bmatrix}
   = \begin{bmatrix}
    I_{m_l} & 0_{m_l \times m_u}\\
    0_{m_s \times m_l} & -\Delta(L_t)\tilde{K} C(L_t) T_p^+
   \end{bmatrix}
   g(\theta_t)^{-1} \frac{d \eta_t}{d t}, \label{eq:ODE_SOS_relaxation}
\end{align}
where $m_s = \mathrm{lt}(z_1)$,$\theta_t^\top = [\theta_{l,t}^\top, \theta_{u,t}^\top]^\top$, $\theta_{u,t} = -T_p \vecs(L_t L_t^\top)$,  and $L_t = \vecl^{-1}(\nu_t)$.
\end{prop}
\begin{proof}
  From Lemma \ref{lem:d_nu_dt_explicit}, with $n=z_1, \xi=-\vecs(\tfrac{d P_t}{d t})$ where $\tfrac{d P_t}{d t}$ given by \eqref{eq:dP_dt}, we arrived at the desired result.
\end{proof}

Using Proposition \ref{prop:theta_l_nu_ode}, we can use a standard ODE solver to solve \eqref{eq:ODE_SOS_relaxation} between two measurement sampling times; $t \in (k \Delta t, (k+1)\Delta t)$. At the end of the propagation time $t = (k+1)\Delta t = t_{k+1}$, we recover $P_{k+1} = P_{t_{k+1}}$ from $\nu_{t_{k+1}}$ via $P_{k+1} = \mathscr{L}^{-1} \circ \vecl^{-1} (\nu_{t_{k+1}})$, and then we recover $\theta_{u,k+1}$ via $\theta_{u,k+1} = - T_p \vecs(P_{k+1})$. After obtaining $\theta_{k+1} = \begin{bmatrix}\theta_{l,k+1}^\top & \theta_{u,k+1}^\top\end{bmatrix}^\top$, we perform the Bayesian update step. Since the exponential family considered here is assumed to be conjugate with respect to the likelihood, then the Bayesian update follows the exact procedure given in \cite{emzir2025}.
\begin{rem}
  Notice that the ODE \eqref{eq:ODE_SOS_relaxation} ensures that $P_t = \mathscr{L}^{-1} \circ \vecl^{-1}(\nu_t)$ to be always positive definite. Hence $\theta_{u,t}$ always corresponds to negative of a sum-of-square form, which ensures that $\theta_t \in \Theta_{SOS} \subset \Theta$ at any time. Unlike the semi-infinite-programming procedure we have developed previously \cite{emzir2025b}, this method does not require an evaluation of optimization method during the propagation of the natural parameter vectors' ODE.
\end{rem}
\subsection{Gaussian Case}\label{sec:Numerical_Examples}
To illustrate the practical application of our framework, we now specialize the general ODE \eqref{eq:ODE_SOS_relaxation} to the Gaussian case. For the Gaussian case, we define $c_l = x$ and $c_u = \vecs(x x^\top) = w_2$; i.e., $r=1$. As we mentioned earlier, in this case, $\mathcal{P}_{[x],r} = \Sigma_{[x],r}$, which means the SOS solution is exactly the same as the semi-infinite programming solution.
This in turn makes $m_l=d, m_u = \mathrm{lt}(d)$. The natural parameter vectors are
\begin{equation}
  \theta_l = \Sigma^{-1}\mu, \qquad \theta_u = -\tfrac{1}{2}K\vecs(\Sigma^{-1}); \label{eq:theta_Gaussian}
\end{equation}
see also \cite[Section 3.2]{malago2015}. The SOS matrix is $P = \frac{1}{2}\Sigma^{-1}$, and since $w_2 = w_{1^2,\mathrm{sym}}$ we have $T=I_{\mathrm{lt}(d)}$ and $T_p = K$.
In Gaussian case, we can explicitly compute the Fisher metric $g(\theta)$ from $\mu,\Sigma$. In what follows we give a compact form of $g(\theta)$ for Gaussian density using Kronecker multiplication. Related derivation can also be found in \cite[Section 3.5]{malago2015}.
\begin{lem}
  The Fisher metric $g(\theta)$ with $c_l = x$ and $c_u = \vecs(x x^\top)$ and $\theta$ given by \eqref{eq:theta_Gaussian} can be written as the following block matrix
  \begin{equation}
    g(\theta) = \begin{bmatrix}
      g_{ll} & g_{l u}\\
      g_{lu}^\top & g_{u u}\\
    \end{bmatrix}\label{eq:g_Gaussian}
  \end{equation}
  where (for brevity, we suppress the dependence on $\theta$)
  \begin{subequations}
    \begin{align}
    g_{ll} =& \Sigma,\label{eq:g_ll}\\
    g_{lu} =& \Sigma E_\mu,\label{eq:g_lu}\\
    g_{uu} =& g_{uu}^{(0)} + E_\mu^\top \Sigma E_\mu,\label{eq:g_uu}\\
    g_{uu}^{(0)} =& D_d^+\left( \Sigma \otimes \Sigma \right)\left( I + M_d \right)(D_d^+)^\top,\label{eq:g_uu_0}
    \end{align}
  \end{subequations}
and $E_\mu = ( I_d \otimes \mu^\top + \mu^\top \otimes I_d )(D_d^+)^\top$.
\end{lem}
\begin{proof}
    This result follows from Isserlis' theorem and the identity
  $g(\theta) = \mathrm{Cov}\left(\begin{bmatrix} x\\ \vecs(x x^\top) \end{bmatrix}\right)$.
  It is obvious that $g_{ll} = \Sigma$. Since $\vecs(x x^\top) = D_d^+ \mathrm{vec}(x x^\top)
  = D_d^+ (x \otimes x)$, the cross-covariance is
  \begin{equation*}
    g_{lu} = \mathrm{Cov}[x,\, D_d^+(x \otimes x)]
    = \mathrm{Cov}[x,\, x \otimes x]\,(D_d^+)^\top.
  \end{equation*}
  By Isserlis' theorem, $\mathrm{Cov}[x_i, x_p x_q] = \mu_p \Sigma_{iq} + \mu_q \Sigma_{ip}$,
  which in Kronecker form gives
  $\mathrm{Cov}[x,\, x \otimes x] = \Sigma(I_d \otimes \mu^\top + \mu^\top \otimes I_d)$.
  Hence $g_{lu} = \Sigma(I_d \otimes \mu^\top + \mu^\top \otimes I_d)(D_d^+)^\top = \Sigma E_\mu$.
  Similarly, the covariance of the quadratic statistics is
  \begin{equation*}
    g_{uu} = D_d^+\, \mathrm{Cov}[x \otimes x]\, (D_d^+)^\top.
  \end{equation*}
  By Isserlis' theorem, $\mathrm{Cov}[(x\otimes x)_{(p,q)},\, (x\otimes x)_{(p',q')}]$
  decomposes into a mean-dependent part and a zero-mean part
  $\Sigma_{pp'}\Sigma_{qq'} + \Sigma_{pq'}\Sigma_{qp'}
  = [(\Sigma \otimes \Sigma)(I + M_d)]_{(p,q),(p',q')}$,
  yielding $g_{uu}^{(0)} = D_d^+(\Sigma \otimes \Sigma)(I + M_d)(D_d^+)^\top$.
  The mean-dependent terms collect into $E_\mu^\top \Sigma E_\mu$,
  giving \eqref{eq:g_uu}.
\end{proof}
\begin{lem}
  The Fisher metric \eqref{eq:g_Gaussian} possesses the following properties:
  \begin{enumerate}
    \item The Schur complement of $g_{ll}$ in $g(\theta)$  simplifies to:
  \begin{equation}
    g(\theta)/g_{ll} = g_{uu}^{(0)}.
  \end{equation}
  \item The Fisher metric's inverse is given by
  \begin{equation}
    g(\theta)^{-1} = \begin{bmatrix}
    \Sigma^{-1} + E_\mu (g_{uu}^{(0)})^{-1} E_\mu^\top & -E_\mu (g_{uu}^{(0)})^{-1} \\
    -(g_{uu}^{(0)})^{-1} E_\mu^\top & (g_{uu}^{(0)})^{-1}
  \end{bmatrix} \label{eq:fisher_inverse}
  \end{equation}
  \end{enumerate}
\end{lem}
\begin{proof}
Using $g_{lu} = \Sigma E_\mu$ and $g_{uu} = g_{uu}^{(0)} + E_\mu^\top \Sigma E_\mu$, we obtain
  $
  g/g_{ll} = g_{uu} - g_{ul} g_{ll}^{-1} g_{lu} = g_{uu}^{(0)}.
  $
Eq. \eqref{eq:fisher_inverse} follows from the Schur complement form.
\end{proof}
We finally state our last result that specializes Proposition \ref{prop:theta_l_nu_ode} for the Gaussian family.
\begin{prop}\label{prop:theta_l_nu_ode_Gaussian}
Let $\theta_k$ be the natural parameters at time $k \Delta t$, for some $k \in \mathbb{N}$, and measurement sampling time $\Delta t$. Then the sum-of-square relaxation of the natural parameters $\theta_t$ between two consecutive measurement sampling times $t \in (k \Delta t, (k+1)\Delta t)$ can be written compactly as
\begin{align}
  \footnotesize
  \dfrac{d}{dt} \begin{bmatrix}
    \theta_{l,t}\\
    \nu_t
  \end{bmatrix}
   = \begin{bmatrix}
    \Sigma^{-1} \frac{d \eta_l}{dt} + E_\mu (g_{uu}^{(0)})^{-1}\left( E_\mu^\top \frac{d \eta_l}{dt} - \frac{d \eta_u}{dt}  \right)\\
    -\Delta(L_t)\tilde{K}C(L_t)K^{-1}(g_{uu}^{(0)})^{-1}\left( \frac{d \eta_u}{dt} - E_\mu^\top \frac{d \eta_l}{dt}\right)
   \end{bmatrix}
 , \label{eq:ODE_SOS_relaxation_Gaussian}
\end{align}
where $\eta_l = \E_{\theta}[c_l], \eta_u = \E_{\theta}[c_u]$.
\end{prop}
\begin{proof}
  Substituting $g(\theta)^{-1}$ given in \eqref{eq:fisher_inverse} to \eqref{eq:ODE_SOS_relaxation}, and $T_p=K$ yields \eqref{eq:ODE_SOS_relaxation_Gaussian}.
\end{proof}

\section{Numerical Example}
In this section, we apply the proposed method to propagate the projected Fokker--Planck equation corresponding to the SDE for reentry model described in \cite{sarkka2007}. We set the fifth-state which is the aerodynamic ballistic coefficient parameter to be constant.  We then normalize the state such that they become dimensionless. The dynamic is given by
\begin{equation}
  \footnotesize
  dx_t =  \begin{bmatrix}
    (V_0/R_0) x_{3,t} \\ (V_0/R_0) x_{4,t} \\ D(x_t)  x_{3,t} + (G(x_t) R_0 / V_0) x_{1,t} \\ D(x_t) x_{4,t} + (G(x_t) R_0 / V_0) x_{2,t}
  \end{bmatrix} dt + \begin{bmatrix}
    0 & 0\\
    0 & 0\\
    \sigma_1/V_0 & 0\\
    0 & \sigma_2/V_0
  \end{bmatrix} dW_t.\label{eq:sde_tracking}
\end{equation}
where the drag $D(x)$, gravity $G(x)$, dimensionless distance $\hat{R}$, and dimensionless velocity $\hat{V}$ are given by
\begin{align*}
  \footnotesize
  &D(x) =b \exp(\frac{R_0}{H_0}(1 - \hat{R}(x))) \cdot V_0 \hat{V},
  \hat{V} = \sqrt{x_3^2 + x_4^2}\\
  &\hat{R}(x) = \sqrt{x_1^2 + x_2^2},
  G(x) = -Gm_0 / (R_0^3 \hat{R}(x)^3).
\end{align*}
We use the same parameters as described in \cite{sarkka2007}, except we set the initial covariance for position to be 100 km$^2$, and velocity covariance to be $10^{-2}$ (km/s)$^2$. We propagate \eqref{eq:ODE_SOS_relaxation_Gaussian} using Gauss--Kronrod sparse-grid level 5. The evolution of $\nu_t$ is shown in Figure \ref{fig:Reentry_4D}. The corresponding $\theta_{u,t}$ is guaranteed to correspond to a second order negative form and is given by $\theta_{u,t}=-K\vecs(L_t L_t^\top)$, where $L_t=\vecl^{-1}(\nu_t)$.
\begin{figure}[H]
  \includegraphics[width=\linewidth]{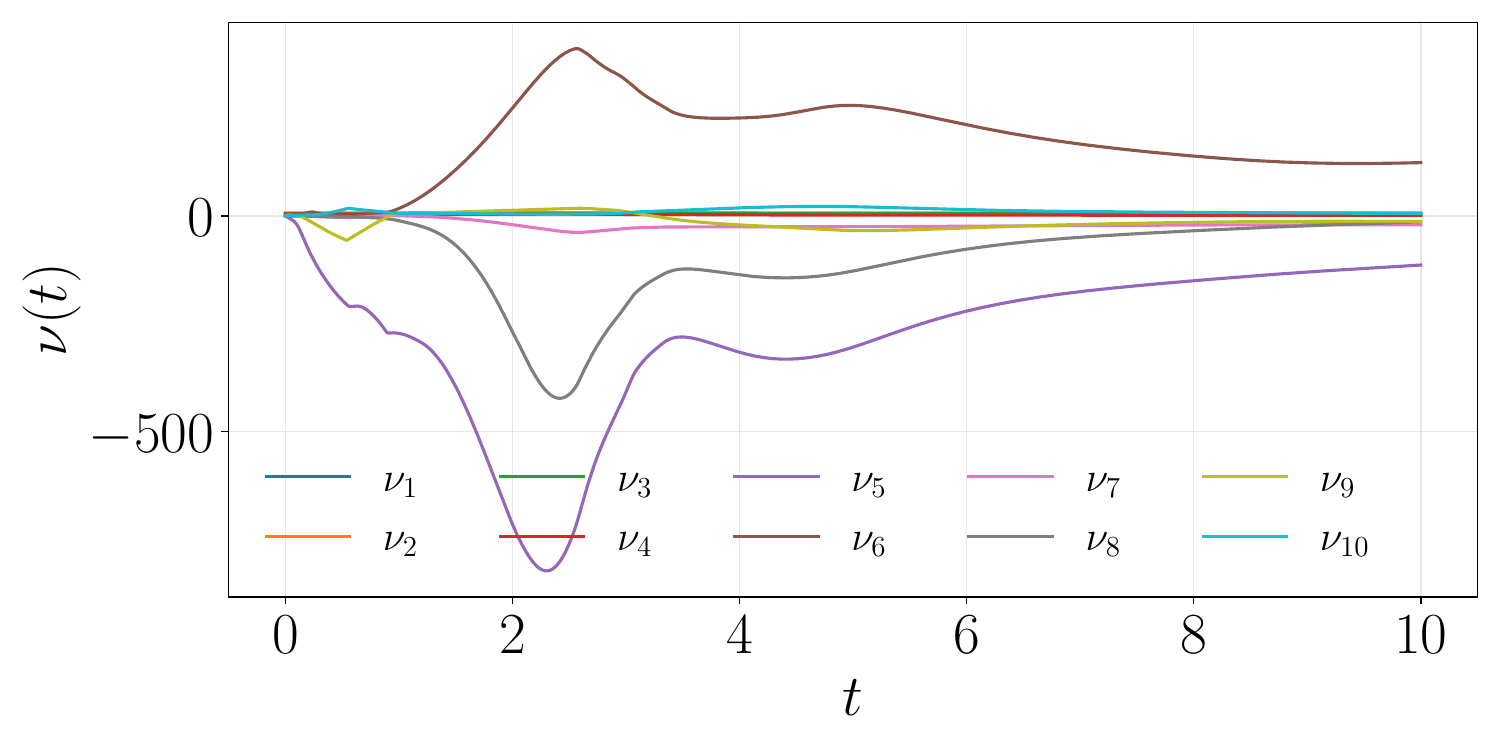}
    \caption{Evolution of $\nu_t$ for the reentry model \eqref{eq:sde_tracking} using SOS-relaxed ODE \eqref{eq:ODE_SOS_relaxation_Gaussian}.}
    \label{fig:Reentry_4D}
\end{figure}

\section{Conclusions}
In this paper, we have presented a sum-of-squares relaxation framework for continuous-discrete projection filters that ensures natural parameters remain within the feasible set by leveraging the log-Cholesky map. We derived explicit, computationally efficient ODEs for the general exponential family, and specialized the result for Gaussian family. This provides a theoretical foundation for stable filter propagation that avoids the need for online optimization.  Future work includes applying this framework to non-Gaussian exponential families with numerical validation, and investigating the tightness of the SOS relaxation for higher-dimensional state spaces.
\footnotesize

\end{document}